\documentclass[11pt,oneside]{amsart}

\usepackage{fullpage}

\usepackage{amsthm,amsfonts,amssymb,amsmath,amsxtra}
\usepackage[all]{xy}
\SelectTips{cm}{}
\usepackage{tikz}
\usepackage{verbatim}
\usepackage{mathrsfs}

\RequirePackage{xspace}
\RequirePackage{etoolbox}
\RequirePackage{varwidth}
\RequirePackage{enumitem}
\RequirePackage{tensor}
\RequirePackage{mathtools}
\RequirePackage{longtable}
\RequirePackage{multirow}
\RequirePackage{makecell}
\RequirePackage{bigints}

\usepackage[backref=page]{hyperref} %

\newcommand{\Ch}{{\mathrm{Ch}}}
\DeclareMathOperator{\charac}{char}

\newcommand{\cl}{{\mathrm{cl}}}

\DeclareMathOperator{\Gal}{Gal}

\DeclareMathOperator{\Gr}{Gr}

\newcommand{\id}{\ensuremath{\mathrm{id}}\xspace}

\DeclareMathOperator{\im}{im}

\DeclareMathOperator{\rank}{rank}

\DeclareMathOperator{\Spec}{Spec}

\newcommand{\U}{\mathrm{U}}

\newcommand{\ov}{\overline}

\newcommand{\lra}{\longrightarrow}

\newenvironment{altenumerate}
   {\begin{list}
      {(\theenumi) }
      {\usecounter{enumi}
       \setlength{\labelwidth}{0pt}
       \setlength{\labelsep}{0pt}
       \setlength{\leftmargin}{0pt}
       \setlength{\itemsep}{\the\smallskipamount}
       \renewcommand{\theenumi}{\roman{enumi}}
      }}
   {\end{list}}
\newenvironment{altitemize}
   {\begin{list}
      {$\bullet$}
      {\setlength{\labelwidth}{0pt}
	   \setlength{\itemindent}{5pt}
       \setlength{\labelsep}{5pt}
       \setlength{\leftmargin}{0pt}
       \setlength{\itemsep}{\the\smallskipamount}
      }}
   {\end{list}}

\setitemize[0]{leftmargin=*,itemsep=\the\smallskipamount}
\setenumerate[0]{leftmargin=*,itemsep=\the\smallskipamount}

\renewcommand{\to}{%
   \ifbool{@display}{\longrightarrow}{\rightarrow}%
   }
\let\shortmapsto\mapsto
\renewcommand{\mapsto}{%
   \ifbool{@display}{\longmapsto}{\shortmapsto}%
   }
\newlength{\olen}
\newlength{\ulen}
\newlength{\xlen}
\newcommand{\xra}[2][]{%
   \ifbool{@display}%
      {\settowidth{\olen}{$\overset{#2}{\longrightarrow}$}%
       \settowidth{\ulen}{$\underset{#1}{\longrightarrow}$}%
       \settowidth{\xlen}{$\xrightarrow[#1]{#2}$}%
       \ifdimgreater{\olen}{\xlen}%
          {\underset{#1}{\overset{#2}{\longrightarrow}}}%
          {\ifdimgreater{\ulen}{\xlen}%
             {\underset{#1}{\overset{#2}{\longrightarrow}}}
             {\xrightarrow[#1]{#2}}}}%
      {\xrightarrow[#1]{#2}}
   }
\makeatother
\newcommand{\xyra}[2][]{%
   \settowidth{\xlen}{$\xrightarrow[#1]{#2}$}%
   \ifbool{@display}%
      {\settowidth{\olen}{$\overset{#2}{\longrightarrow}$}%
       \settowidth{\ulen}{$\underset{#1}{\longrightarrow}$}%
       \ifdimgreater{\olen}{\xlen}%
          {\mathrel{\xymatrix@M=.12ex@C=3.2ex{\ar[r]^-{#2}_-{#1} &}}}%
          {\ifdimgreater{\ulen}{\xlen}%
             {\mathrel{\xymatrix@M=.12ex@C=3.2ex{\ar[r]^-{#2}_-{#1} &}}}
             {\mathrel{\xymatrix@M=.12ex@C=\the\xlen{\ar[r]^-{#2}_-{#1} &}}}}}%
      {\mathrel{\xymatrix@M=.12ex@C=\the\xlen{\ar[r]^-{#2}_-{#1} &}}}%
   }
\makeatletter
\newcommand{\xla}[2][]{%
   \ifbool{@display}%
      {\settowidth{\olen}{$\overset{#2}{\longleftarrow}$}%
       \settowidth{\ulen}{$\underset{#1}{\longleftarrow}$}%
       \settowidth{\xlen}{$\xleftarrow[#1]{#2}$}%
       \ifdimgreater{\olen}{\xlen}%
          {\underset{#1}{\overset{#2}{\longleftarrow}}}%
          {\ifdimgreater{\ulen}{\xlen}%
             {\underset{#1}{\overset{#2}{\longleftarrow}}}
             {\xleftarrow[#1]{#2}}}}%
      {\xleftarrow[#1]{#2}}
   }
\newcommand{\isoarrow}{%
   \ifbool{@display}{\overset{\sim}{\longrightarrow}}{\xrightarrow\sim}%
   }
\renewcommand{\lra}{%
   \ifbool{@display}{\longleftrightarrow}{\leftrightarrow}%
   }
\newcommand{\HOM}{\mathcal{H}om}

\newcommand{\Tate}{\mathrm{Tate}}

\DeclareFontFamily{U}{matha}{\hyphenchar\font45}
\DeclareFontShape{U}{matha}{m}{n}{
      <5> <6> <7> <8> <9> <10> gen * matha
      <10.95> matha10 <12> <14.4> <17.28> <20.74> <24.88> matha12
      }{}
\DeclareSymbolFont{matha}{U}{matha}{m}{n}
\DeclareFontFamily{U}{mathx}{\hyphenchar\font45}
\DeclareFontShape{U}{mathx}{m}{n}{
      <5> <6> <7> <8> <9> <10>
      <10.95> <12> <14.4> <17.28> <20.74> <24.88>
      mathx10
      }{}
\DeclareSymbolFont{mathx}{U}{mathx}{m}{n}

\DeclareMathSymbol{\obot}         {2}{matha}{"6B}

\newtheorem{theorem}{Theorem}[section]
\newtheorem{proposition}[theorem]{Proposition}
\newtheorem{lemma}[theorem]{Lemma}

\newtheorem{corollary}[theorem]{Corollary}

\theoremstyle{definition}
\newtheorem{definition}[theorem]{Definition}
\newtheorem{example}[theorem]{Example}

\newtheorem{remark}[theorem]{Remark}

\numberwithin{equation}{theorem}

\newcommand{\kb}{{\bar k}}

\DeclareMathOperator{\HH}{H}
\DeclareMathOperator{\codim}{codim}

\DeclareMathOperator{\DL}{DL}
\newcommand{\XV}{\DL(V)}
\newcommand{\XVb}{\DL(V)_{\kb}}
\newcommand{\ZW}{Z(W)}
\newcommand{\ZWp}{Z(W')}
\newcommand{\XW}{\DL(W/W^\perp)}

\newcommand{\oU}{\ov U}
\renewcommand{\k}{k_0}

\newcommand{\GrV}{\Gr_{d+1}(V)}

\newcommand{\Grd}{\Gr_{d}(V)}

\newcommand{\Grdd}{\Gr_{d,d+1}(V)}
\newcommand{\SGr}{\mathcal{S}_{\GrV}}
\newcommand{\QGr}{\mathcal{Q}_{\GrV}}

\newcommand{\PV}{\mathbb{P}(\wedge^{d+1} V)}
\DeclareMathOperator{\pl}{Pl}

\newcommand{\UX}{\mathcal{U}_{\XV}}
\newcommand{\UW}{\mathcal{U}_{\ZW}}
\newcommand{\LX}{\mathcal{L}_{\XV}}
\newcommand{\LW}{\mathcal{L}_{\ZW}}
\newcommand{\LXW}{\mathcal{L}_{\XW}}

\newcommand{\VX}{\mathcal{V}_{\XV}}
\newcommand{\VW}{\mathcal{W}_{\ZW}}
\newcommand{\QX}{\mathcal{Q}_{\XV}}
\newcommand{\QW}{\mathcal{Q}_{\ZW}}

\newcommand{\Ell}{l}

\usepackage[enableskew,vcentermath]{youngtab}

\usepackage[boxsize=1em]{ytableau}

\title{Degrees of unitary Deligne--Lusztig varieties}

\author[Chao Li]{Chao Li}
\address{Columbia University, Department of Mathematics, 2990 Broadway,	New York, NY 10027, USA}
\email{chaoli@math.columbia.edu} 

\date{\today}

\begin{document}

\maketitle{}

\begin{abstract}
  We prove an explicit degree formula for certain unitary Deligne--Lusztig varieties. Combining with an alternative degree formula in terms of Schubert calculus, we deduce several algebraic combinatorial identities which may be of independent interest.
\end{abstract}

\section{Introduction}

\subsection{Degrees of unitary Deligne--Lusztig varieties}
Let $\k=\mathbb{F}_q$ be a finite field of size $q$. Let $k=\mathbb{F}_{q^2}$ be the quadratic extension of $\k$.  Let $V$ be a (nondegenerate) $k/\k$-hermitian space of dimension $n$. Let $\Gr_{m}(V)$ be the Grassmannian of $m$-dimensional subspaces of $V$, which is a smooth projective variety over $k$ of dimension $m(n-m)$.

 We assume that $V$ has dimension $n=2d+1$ ($d\ge0$) over $k$ and take $m=d+1$. Define $\XV\subseteq \Gr_{d+1}(V)$ to be the closed $k$-subscheme parametrizing $(d+1)$-dimensional subspaces $U$ such that $U^\perp\subseteq U$, where $U^\perp$ is the orthogonal complement of $U$ in $V$ (see the more precise Definition \ref{def:DLV}). It is smooth, projective and geometrically irreducible of dimension $d$, and is known as a \emph{unitary Deligne--Lusztig variety}, as it can be identified as a generalized Deligne--Lusztig variety associated to a parabolic subgroup of the odd unitary group $\U(V)$ (\cite[\S 4.5]{Vollaard2011}).
 This class of varieties shows up in the study of the supersingular locus of unitary Shimura varieties (\cite{Vollaard2011}) and plays an important role in the arithmetic of unitary Shimura varieties, such as the Kudla--Rapoport conjecture (\cite{Kudla2011,Kudla2014}).
 
Our first main result is a simple degree formula for $\XV$. Recall that the \emph{degree} of  a projective variety $X\subseteq \mathbb{P}^N$ of pure dimension $d$  is given by the geometric intersection number of $X$ with $d$ general hyperplanes in $\mathbb{P}^N$. Denote by $\deg \XV$ the degree of $\XV$ under the Pl\"ucker embedding $\Gr_{d+1}(V)\hookrightarrow \mathbb{P}(\wedge^{d+1}V)$. 

\begin{theorem}\label{thm:mainA} Let $V$ be a $k/\k$-hermitian space of dimension $n=2d+1$. Then
  \begin{equation}
  \deg \XV=\prod_{i=1}^d{\frac{1-q^{2i}}{1-q}}.\label{eq:main}
\end{equation}
\end{theorem}

\begin{remark}
Here (and below) we regard the empty product (when $d=0$) as 1. The right hand side of \eqref{eq:main} can also be interpreted as the $q$-analogue of the double factorial $$[2d]_q!!:=[2d]_q [2d-2]_q\cdots [2]_q,$$ where $[n]_q=\frac{1-q^n}{1-q}$  is the $q$-analogue of $n$.
\end{remark}

Theorem \ref{thm:mainA} will be proved in \S\ref{sec:proof-theor-refthm:m}.  Its proof is inspired by the \emph{higher local modularity} in the recent proof of the Kudla--Rapoport conjecture (cf. \cite[\S 6.4]{LZ}). In fact, the key formulas (Propositions \ref{prop:c1}, \ref{prop:c1d}) can be extracted from certain vertical intersection formulas (\cite[Lemmas~6.4.5, 6.4.6]{LZ}) on unitary Rapoport--Zink spaces. To make the ideas more transparent, here we work directly on $\XV$ and introduce the notions of \emph{special cycles} $\ZW\subseteq\XV$  (Definition \ref{def:specialcycle}) and a \emph{special line bundle} $\LX$ (Definition \ref{def:speciallinebundle}) on $\XV$. These notions may be viewed respectively as finite field analogues of special cycles and tautological line bundles on unitary Shimura varieties. From this perspective, the degree formula in Theorem \ref{thm:mainA}, or more precisely the formula for $c_1(\LX^*)^d$ in Proposition \ref{prop:c1d}, may be viewed as a finite field analogue of the constant term formula in Kudla's geometric Siegel--Weil formula (relating the geometric volume of unitary Shimura varieties and an abelian $L$-value, cf. \cite[(4.4)]{Kudla2004} for the analogue for orthogonal Shimura varieties).

  The proof of Theorem \ref{thm:mainA} ultimately relies on identifying the special cycles $\ZW\subseteq \XV$ as unitary Deligne--Lusztig subvarieties and the (proved) Tate conjecture for $\XV$, in order to perform induction on the dimension $d$. These inductive structures are available for Deligne--Lusztig varieties beyond those of unitary types (e.g., the type $^2D_n$ considered in \cite[\S7.6]{LZ22} and two other types $B_n, C_n$ listed in \cite[\S3]{He2019}), and it would be interesting to extend the method to obtain degree formulas for more general Deligne--Lusztig varieties.

  \subsection{Schubert calculus and applications to algebraic combinatorics} We also prove a different formula for $\deg\XV$ in terms of Schubert calculus on Grassmannians (see also the related general works \cite{Kim20, HP20}).

  \begin{theorem}\label{thm:mainB}
    The following identity holds in $\Ch^{d(d+1)}(\GrV_\kb)_\mathbb{Q}\simeq \mathbb{Q}$:
    \begin{equation}
      \label{eq:schubertmain}
      \deg\XV=\sum_c \sigma_c\sigma_{\widehat c'}\sigma_1^d q^{|c|},
    \end{equation}
    where the sum runs over all integer tuples $c=(c_1,\ldots,c_d)$ with $d\ge c_1\ge \cdots \ge c_d\ge0$. Here (as recalled in \S\ref{sec:remind-schub-calc}):
    \begin{altitemize}
    \item     $|c|=c_1+\cdots +c_d$,
    \item $\sigma_c\in\Ch^{|c|}(\GrV_\kb)_\mathbb{Q}$ is the \emph{Schubert class},
    \item $\widehat c$ is the \emph{complement} of $c$ defined by $\hat c:=(d-c_d,\ldots, d-c_1)$,
    \item $\widehat c'$ is the \emph{conjugate} of $\widehat c$.
    \end{altitemize}

  \end{theorem}

  Combining Theorems \ref{thm:mainA} and \ref{thm:mainB} we deduce several combinatorial identities.
  \begin{corollary}\label{cor:main}
    \begin{altenumerate}
    \item\label{item:m0} The right hand side of (\ref{eq:schubertmain}) is equal to $\prod_{i=1}^d{\frac{1-q^{2i}}{1-q}}$.
    \item\label{item:m2} The $x_1^{2d}x_2^{2d-1}\cdots x_d^{d+1}$-coefficient of $$\left(\prod_{i,j=1}^d(qx_i+x_j)\right)(x_1+\cdots+x_d)^d\left(\prod_{i<j}(x_i-x_j)\right)$$ is equal to $\prod_{i=1}^d{\frac{1-q^{2i}}{1-q}}$.      
    \item\label{item:m1} Let $\Ell\ge0$. The following two sets have the same size:
      \begin{altitemize}
      \item The set of standard Young tableaux of skew shape $(\widehat c')^*/c$, where $c$ runs over $(c_1,\ldots,c_d)$ such that $d\ge c_1\ge \cdots \ge c_d\ge0$ and $|c|=\Ell$. Here $(\widehat c')^*$ is the \emph{dual} of $\widehat c'$ (recalled in \S\ref{sec:remind-schub-calc}).
      \item The set of  ordered partitions $(\Ell_1,\ldots, \Ell_d)$ of $\Ell$ satisfying  $\Ell=\Ell_1+\cdots+\Ell_d$ and $0\le \Ell_i\le 2i-1$ for all $1\le i\le d$.
      \end{altitemize}
    \end{altenumerate}
    \end{corollary}

Theorem \ref{thm:mainB} and Corollary \ref{cor:main} will be proved \S\ref{sec:proof-theor-refthm:m-1} and \S\ref{sec:proof-coroll-refc} respectively. We notice that the combinatorial identities in Corollary \ref{cor:main} seem to be quite nontrivial even for small values of $d$ (see Example \ref{exa:smalld}). It would be very interesting to find more direct combinatorial proofs of these identities.

\subsection{Acknowledgments} It is the author's pleasure to dedicate this paper to Steve Kudla on the occasion of his 70th birthday. The influence of his original insights on the geometric and arithmetic Siegel--Weil formula on this paper should be evident to the readers. The author is also grateful to Z. Yun and W. Zhang for helpful conversations, and to M. Rapoport and the anonymous referee for useful comments. The author's work is partially supported by the NSF grant DMS-2101157.

\subsection{Notation}\label{sec:notations}

Let $X$ be a smooth projective variety over a finite
field $k$. Let $\Ch^r(X_\kb)$ be the Chow group of codimension $r$ algebraic cycles of $X_\kb$ defined modulo rational equivalence. Fix a prime $\ell\ne \charac(k)$, denote by $$\cl_r: \Ch^r(X_\kb)_\mathbb{Q}\rightarrow \HH^{2r}(X_\kb, \mathbb{Q}_\ell)(r)$$ the $\ell$-adic cycle class map.
Denote by $\Tate^{2r}_\ell(X_\kb)\subseteq \HH^{2r}(X_\kb,\mathbb{Q}_\ell)(r)$ the subspace of Tate classes, i.e., the elements fixed by an open subgroup of $\Gal(\kb/k)$.  Then $\cl_r$ intertwines the intersection product $\cdot$ on the Chow ring and the cup product $\cup$ on the cohomology ring, namely the following diagram commutes,
\begin{equation}
  \label{eq:chcl}
\begin{gathered}
  \xymatrix@C=0.3em{\Ch^r(X_\kb)_\mathbb{Q} \ar[d]^{\cl_r} &\times & \Ch^s(X_\kb)_\mathbb{Q} \ar[d]^{\cl_s} \ar[rr]^-{\cdot} && \Ch^{r+s}(X_\kb)_\mathbb{Q} \ar[d]^{\cl_{r+s}} \\ \HH^{2r}(X_\kb,\mathbb{Q}_\ell)(r) &\times & \HH^{2s}(X_\kb, \mathbb{Q}_\ell)(s) \ar[rr]^-{\cup} &&  \HH^{2(r+s)}(X_\kb, \mathbb{Q}_\ell)(r+s).}
\end{gathered}
\end{equation}  When $r=1$, the cycle class map $\cl_1:\Ch^1(X_\kb)_\mathbb{Q}\rightarrow \HH^2(X_\kb,\mathbb{Q}_\ell)(1)$ is injective.
When $r=\dim X$, we often identify $\Ch^{\dim X}(X_\kb)_\mathbb{Q}\simeq \mathbb{Q}$ via the degree isomorphism (cf. \cite[Example 1.6.6]{Ful98}). Recall that the Tate conjecture (\cite[Conjecture 1]{Tate1965}, or \cite[Conjecture $T^i$]{Tate1994}) asserts that for any $r\ge0$, the image of $\ell$-adic cycle class map $\cl_r$ $\mathbb{Q}_\ell$-spans $\Tate^{2r}_\ell(X_\kb)$.

For a subvariety $Z\subseteq X$ of pure codimension $r$, we denote by $[Z]_X$ its class in $\Ch^r(X_\kb)$. When the ambient variety $X$ is clear we suppress the subscript and simply write $[Z]$.   For any vector bundle $\mathcal{V}$ on $X$, denote by $c_r(\mathcal{V})\in \Ch^r(X_\kb)$ its $r$-th Chern class. By abusing notation we also denote by the same symbols $[Z]_X$, $[Z]$ and $c_r(\mathcal{V})$ their images in $\Ch^r(X_\kb)_\mathbb{Q}$. For any vector bundle $\mathcal{V}$ on $X$, denote by $\mathcal{V}^*:=\HOM_{\mathcal{O}_X}(\mathcal{V},\mathcal{O}_X)$ its dual bundle.

\section{Degrees of unitary Deligne--Lusztig varieties}

Let $p$ be a prime and let $q$ be a power of $p$. Let $\k=\mathbb{F}_q$ be a finite field of size $q$. Let $k=\mathbb{F}_{q^2}$ be the quadratic extension of $\k$. Denote by $\sigma$ the absolute $q$-Frobenius endomorphism on any scheme of characteristic $p$.

\subsection{Hermitian spaces}
Let $V$ be a (nondegenerate) \emph{$k/\k$-hermitian space} of dimension $n$, i.e., a $k$-vector space of dimension $n$ equipped with a (nondegenerate) paring $(\ ,\ ): V\times V\rightarrow k$ that is linear in the first variable, $\sigma$-linear in the second variable and satisfies
\begin{equation}
  \label{eq:symmetry}
  (x,y)=(y,x)^\sigma
\end{equation}
for any $x,y\in V$. For any $k$-subspace $U\subseteq V$, denote by $$U^\perp=\{x\in V: (x,U)=0\}$$ its orthogonal complement.

More generally, for any $k$-scheme $S$, put $V_S:=V \otimes_{k}\mathcal{O}_S$. Then there is a unique pairing $(\ ,\ )_S: V_S\times V_S\rightarrow \mathcal{O}_S$  extending $(\ ,\ )$ that is $\mathcal{O}_S$-linear in the first variable and $(\mathcal{O}_S,\sigma)$-linear in the second variable, given by $$(x \otimes \lambda, y\otimes\mu)=\lambda \mu^\sigma (x,y)$$ for any $x,y\in V$ and $\lambda,\mu\in \mathcal{O}_S$. For any subbundle $U\subseteq V_S$, we denote by $$U^\perp=\{x\in V_S: (x,U)=0\}$$ its left orthogonal complement. (Notice that unlike the case $S=\Spec k$, in general the symmetry \eqref{eq:symmetry} does not necessarily hold for $x,y\in V_S$, and the left orthogonal complement $U^\perp$ does not necessarily agree with the right orthogonal complement.)

\subsection{Unitary Deligne--Lusztig varieties and special cycles} From now on fix a $k/\k$-hermitian space $V$ of dimension $n=2d+1$ ($d\ge0$).

\begin{definition}\label{def:DLV}
 Define the \emph{unitary Deligne--Lusztig variety} $\XV\subseteq \Gr_{d+1}(V)$ to be the closed $k$-subscheme parametrizing $(d+1)$-dimensional subspaces $U$ such that $U^\perp\subseteq U$, i.e., for any $k$-scheme $S$, $$\XV(S)=\{\text{subbundles } U\subseteq V_S: \rank U=d+1,\ U^\perp\subseteq U\}.$$ 
\end{definition}

  The relation between $\XV$ and classical Deligne--Lusztig varieties associated to the unitary group $\U(V)$ can be found in \cite[Corollary 2.17]{Vollaard2010} and the reader can recognize that the Frobenius appearing in the definition of classical Deligne--Lusztig varieties arises from the $\sigma$-linearity in the second variable of the hermitian form $(\ ,\ )$.

\begin{definition}
  Define a \emph{special subspace} $W\subseteq V$ to be a $k$-subspace such that $W^\perp\subseteq W$. If $W$ is a special subspace, then $W/W^\perp$ is a (nondegenerate) $k/\k$-hermitian space under the pairing induced from $V$.
\end{definition}

\begin{definition}\label{def:specialcycle}
   Let $W\subseteq V$ be a   $k$-subspace.  Define the \emph{special cycle} $\ZW\subseteq \XV$ to be the closed subscheme parametrizing subspaces $U$ satisfying $U\subseteq W$, i.e., for any $k$-scheme $S$, $$\ZW(S)=\{U\in \XV(S): U\subseteq W_S \}.$$ By definition it is clear that  
   \begin{altitemize}
   \item $\ZW$ is nonempty only if $W\subseteq V$ is a \emph{special subspace}.     
   \item $Z(V)=\XV$,
   \item $\ZW\cap \ZWp=Z(W\cap W')$ for two subspaces $W,W'\subseteq V$ ($\cap$ denotes the scheme-theoretic intersection).
   \item $\ZW\subseteq \ZWp$ if $W\subseteq W'$.
   \end{altitemize}
 \end{definition}

We summarize several known results on $\XV$ and  $\ZW$ which we will need later.

\begin{proposition}\label{prop:tate}
  \begin{altenumerate}
  \item\label{item:1} $\XV$ is smooth, projective and geometrically irreducible of dimension $d$.
  \item Let $W\subseteq V$ be a special subspace of codimension $r$. Then there is a canonical isomorphism of $k$-schemes
    \begin{equation}
      \label{eq:ZWXW}
      \ZW\simeq\XW.
    \end{equation}
 In particular, $Z(W)$ is smooth, projective and geometrically irreducible, and has codimension $r$ in $\XV$, and we call $Z(W)$ a \emph{codimension $r$ special cycle in $\XV$.}    
  \item 
    The Tate conjecture (see \S\ref{sec:notations}) holds for $\XV$.
  \item\label{item:4} For any $r\ge0$, the space $\Tate_\ell^{2r}(\XVb)$ is spanned by the cycle classes of codimension $r$ special cycles $[\ZW]$, where $W\subseteq V$ runs over all special subspace of codimension $r$.
  \item\label{item:5} For any $r\ge0$, the cup product induces a perfect pairing $$\Tate_\ell^{2r}(\XVb)\times \Tate_\ell^{2d-2r}(\XVb)\xrightarrow{\cdot} \mathbb{Q}_\ell.$$
  \end{altenumerate}
\end{proposition}

\begin{proof}
  \begin{altenumerate}
  \item This is \cite[Lemma 4.5]{Vollaard2011}, or the more general \cite[Proposition A.1.3 (2)]{LTXZZ}. 
  \item Consider the morphism $\ZW\rightarrow \XW$ defined by $$\ZW(S)\rightarrow \XW(S),\quad U\mapsto \oU=U/W^\perp_S$$ for any $k$-scheme $S$. It is an isomorphism,  with inverse given by the morphism defined by $\oU\mapsto W^\perp_S+ \oU$. The rest follows from Item  (\ref{item:1}) since $W/W^\perp$ has $k$-dimension $n-2r$.
  \item This is \cite[Theorem 5.3.2 (i)]{LZ}.
  \item This is a combination of \cite[Theorem 5.3.2 (i)]{LZ} and \cite[Corollary 2.17]{Vollaard2010}.
  \item This is \cite[(6.4.0.4)]{LZ}.
\qedhere
  \end{altenumerate}
\end{proof}

\subsection{Natural vector bundles on $\XV$ and $\ZW$}

\begin{definition}\label{def:speciallinebundle}
 Define
  \begin{altitemize}
  \item   $\VX:=V \otimes_k \mathcal{O}_{\XV}$ the universal bundle on $\XV$,
  \item $\UX\subseteq \VX$ and $\UX^\perp\subseteq\VX$ the two universal subbundles on $\XV$,
  \item $\QX:=\VX/\UX$ the universal quotient bundle on $\XV$,
  \item $\LX:=\UX/\UX^\perp,$ which we call the \emph{special line bundle} on $\XV$.
  \end{altitemize}

  Similarly for a special cycle $\ZW\subseteq\XV$ associated to a special subspace $W\subseteq V$, define
  \begin{altitemize}
  \item $\VW:=(W/W^\perp) \otimes_k \mathcal{O}_{\ZW}$ the universal bundle on $\ZW\simeq\XW$,
  \item $\UW\subseteq \VW$ and $\UW^\perp\subseteq\VW$ the two universal subbundles on $\ZW$,
  \item $\QW:=\VW/\UW$ the universal quotient bundle on $\ZW$,
  \item $\LW:=\UW/\UW^\perp,$ which we call the \emph{special line bundle} on $\ZW$.
  \end{altitemize}
\end{definition}

Denote by $$j_W: \ZW\hookrightarrow \XV$$ the natural closed embedding.  By definition, we have 
\begin{equation}
  \label{eq:jW}
   \LW=j_W^*\LX. 
\end{equation}
Also by definition, the canonical isomorphism \eqref{eq:ZWXW} induces a canonical isomorphism $\LW\simeq \LXW$ and thus
\begin{equation}
  \label{eq:LXLWX}
  j_W^*\LX\simeq \LXW.
\end{equation}

\begin{proposition} Let $W\subseteq V$ be a special subspace of codimension $r$. 
  \begin{altenumerate}
  \item We have a canonical isomorphism
\begin{equation}
  \label{eq:QX}
  j_W^*\QX/\QW\simeq (V/W)\otimes_k \mathcal{O}_{\ZW}.
\end{equation} 
\item Let $\mathcal{N}_{\ZW/\XV}$ be the normal bundle of $\ZW$ in $\XV$. Then we have an isomorphism $$\mathcal{N}_{\ZW/\XV}\simeq \left(\LW^{\oplus r}\right)^*.$$   In particular, when $W\subseteq V$ has codimension 1, we have an isomorphism of line bundles
  \begin{equation}
    \label{eq:linebundle}
    \mathcal{N}_{\ZW/\XV}\simeq \LW^*.
  \end{equation}
  \end{altenumerate}
\end{proposition}

\begin{proof}
  \begin{altenumerate}
  \item The result follows immediately from the definition of $\QX$ and $\QW$.
  \item   By \cite[Proposition A.1.3 (2)]{LTXZZ}, we have canonical isomorphisms for the tangent bundles
  \begin{equation*}
    \mathcal{T}_{\XV/k}\simeq\HOM(\LX, \QX), \quad    \mathcal{T}_{\ZW/k}\simeq\HOM(\LW, \QW).
  \end{equation*} Hence by the exact sequence $$0\rightarrow \mathcal{T}_{\ZW/k}\rightarrow j_W^*\mathcal{T}_{\XV/k}\rightarrow \mathcal{N}_{\ZW/\XV}\rightarrow 0$$ and \eqref{eq:jW}, we obtain a canonical isomorphism $$\mathcal{N}_{\ZW/\XV}\simeq \HOM(\LW, j_W^*\QX/\QW).$$ The result then follows from \eqref{eq:QX} as $(V/W) \otimes_k \mathcal{O}_{\ZW}\simeq \mathcal{O}_{\ZW}^r$. \qedhere
  \end{altenumerate}
\end{proof}

\subsection{Relations with Grassmannians}

For $0\le m\le n$, let $\Gr_{m}(V)$ be the Grassmannian of $m$-dimensional subspaces of $V$, which is a smooth projective variety over $k$ of dimension $m(n-m)$. We have the \emph{Pl\"ucker embedding} $$\pl:\Gr_m(V)\hookrightarrow \mathbb{P}(\wedge^m(V))\simeq \mathbb{P}^{N},\quad N=\textstyle{n\choose m}-1,$$ defined by sending an $m$-dimensional subspace with basis $\{e_1,\ldots,e_m\}$ to the line generated by $e_1\wedge \cdots \wedge e_m\in \wedge^m(V)$ (independent of the choice of the basis). 

\begin{definition}
  Define $\SGr$ (resp. $\QGr$) to be the universal subbundle (resp. universal quotient subbundle) on $\GrV$. Denote by $$i:\XV\hookrightarrow \GrV$$ the natural closed embedding.
\end{definition}

\begin{lemma}\label{lem:bunGr} The following identities holds:
  \begin{altenumerate}
  \item\label{item:b1} $\UX=i^*\SGr$,
  \item\label{item:b2} $\UX^\perp\simeq i^*(\sigma^*\QGr)^*$,
  \item\label{item:b3} $c_1(\SGr)=-c_1(\QGr)$,    
  \item\label{item:b4} $c_1(\LX)=(1-q) i^*c_1(\SGr)$.
  \end{altenumerate}
\end{lemma}

\begin{proof}
  \begin{altenumerate}
  \item It follows from the definition of $\SGr$ and $\UX$.
  \item It follows from Item (\ref{item:b1}) and the nondegenerate pairing $(\ ,\ )$ ($\sigma$-linear in the second variable).
  \item It follows from the defining exact sequence \begin{equation*}
    0\rightarrow \SGr\rightarrow (V \otimes_k \mathcal{O}_{\GrV})\rightarrow \QGr\rightarrow 0.
  \end{equation*}
\item By the definition of $\LX$ together with Item (\ref{item:b1}) and Item (\ref{item:b2}), we have $$c_1(\LX)=c_1(\UX)-c_1(\UX^\perp)=i^*c_1(\SGr)-i^*c_1((\sigma^*\QGr)^*).$$ By Item (\ref{item:b3}), this evaluates to $$i^*c_1(\SGr)-q(i^*c_1(\SGr))=(1-q) i^*c_1(\SGr),$$ as desired. \qedhere
  \end{altenumerate}
\end{proof}

\begin{example}\label{exa:fermat}
  When $d=1$ (i.e., when $\dim V=3$), we have $\GrV\simeq \mathbb{P}^2$ with
  $\QGr\simeq \mathcal{O}_{\mathbb{P}^2}(1)$. Thus $c_1(\SGr)=c_1(\mathcal{O}_{\mathbb{P}^2}(-1))$ by Lemma \ref{lem:bunGr} (\ref{item:b3}). By definition $\XV\subseteq \GrV$ is isomorphic to a Fermat curve of degree $1+q$ (\cite[Remark 4.7]{Vollaard2010}), $$\{[x,y,z]\in \mathbb{P}^2: x^{q+1}+y^{q+1}+z^{q+1}=0\}\subseteq \mathbb{P}^2.$$
Hence by Lemma \ref{lem:bunGr} (\ref{item:b4}),
the following identity holds in $\Ch^1(\XVb)_\mathbb{Q}\simeq \mathbb{Q}$,
 \begin{equation*}
    c_1(\LX)=(1-q)i^*c_1(\mathcal{O}_{\mathbb{P}^2}(-1))=-(1-q)(1+q)=-(1-q^2),
  \end{equation*}
  and so
  \begin{equation}
        \label{eq:fermat}
        c_1(\LX^*)=-c_1(\LX)=1-q^2.
  \end{equation}
\end{example}

\subsection{Chern classes of the special line bundle}

\begin{lemma}\label{lem:c1ZW}
  Let $\ZW\subseteq\XV$ be a 1-dimensional special cycle associated to a special subspace $W\subseteq V$ of codimension $d-1$. Then the following identity holds in $\Ch^{d}(\XVb)_\mathbb{Q}\simeq \mathbb{Q}$: $$c_1(\LX^*)\cdot [\ZW]=1-q^2.$$
\end{lemma}

\begin{proof}
  By the projection formula we know that $$c_1(\LX^*)\cdot [\ZW]=j_{W,*}(c_1(j_W^*\LX)).$$ By \eqref{eq:LXLWX}, we have (under the isomorphism \eqref{eq:ZWXW}) $$c_1(j_{W}^*\LX^*)=c_1(\LXW^*).$$ As $W/W^\perp$ is of dimension 3, we know the latter evaluates to $1-q^2$ in $\Ch^1(\XW_\kb)_\mathbb{Q}\simeq \mathbb{Q}$ by \eqref{eq:fermat}. The result then follows.
\end{proof}

\begin{proposition}\label{prop:c1}
  The following identity holds in $\Ch^1(\XVb)_\mathbb{Q}$:
  \begin{equation}\label{eq:LX}
    c_1(\LX^*)=\frac{1-q^2}{1+q^{2d+1}}\sum_{\codim W=1}[\ZW],
  \end{equation}
here the sum runs over all special subspaces $W\subseteq V$ of codimension 1.
\end{proposition}

\begin{proof}
  When $d=1$, the number of special subspaces $W\subseteq V$ is equal to the number of isotropic lines in the 3-dimensional $k/\k$-hermitian space $V$, which is $1+q^3$ (e.g., by \cite[Lemma 1.9.1]{LZ}). Hence the right hand side of \eqref{eq:LX} evaluates to $$\frac{1-q^2}{1+q^3}\cdot (1+q^3)=1-q^2$$ in $\Ch^1(\XVb)_\mathbb{Q}\simeq \mathbb{Q}$, which agrees with the left hand side by \eqref{eq:fermat}.

 When $d>1$,  by Proposition \ref{prop:tate} (\ref{item:4})~(\ref{item:5}) and the commutativity of \eqref{eq:chcl}, to prove \eqref{eq:LX} it suffices to show that for any 1-dimensional special cycle $\ZWp\subseteq \XV$ (associated to any codimension $d-1$ special subspace $W'\subseteq V$), the following identity holds in $\Ch^d(\XVb)_\mathbb{Q}\simeq \mathbb{Q}$,
  \begin{equation}
    \label{eq:complement}
    c_1(\LX^*)\cdot [\ZWp]=\frac{1-q^2}{1+q^{2d+1}}\sum_{\codim W=1}[\ZW]\cdot[\ZWp].
  \end{equation}

  For the terms of the right hand side of \eqref{eq:complement}, we have three cases.
  \begin{altenumerate}
  \item\label{item:a} When  $W'\subseteq W$, we have $\ZWp\subseteq\ZW$. By the excess intersection formula \cite[Corollary 6.3]{Ful98}, we know that $$[\ZW]\cdot[\ZWp]=j_{W,*}(c_1(\mathcal{N}_{\ZW/\XV})\cdot [\ZWp]_{\ZW}).$$
    By \eqref{eq:linebundle} and \eqref{eq:LXLWX}  we have (under the isomorphism \eqref{eq:ZWXW}) $$c_1(\mathcal{N}_{\ZW/\XV})\cdot [\ZWp]_{\ZW}=c_1(\LXW^*)\cdot [\ZWp]_{\XW},$$ which evaluates to $1-q^2$ in $\Ch^{d-1}(\XW_\kb)_\mathbb{Q}\simeq \mathbb{Q}$ by Lemma \ref{lem:c1ZW} applied to the 1-dimensional special cycle $\ZWp\subseteq \XW$. Thus in this case $$[\ZW]\cdot [\ZWp]=1-q^2.$$
  \item\label{item:b} When $W'\not\subseteq W$ and $W'\cap W\subseteq V$ is a special subspace, we know that $W'\cap W$ has codimension $d$ in $V$ and thus scheme-theoretic intersection $\ZW\cap\ZWp=Z(W'\cap W)$ is isomorphic to $\Spec k$ by Proposition \ref{prop:tate} (\ref{item:1}), and thus $$[\ZW]\cdot[\ZWp]=1.$$
  \item When $W'\not\subseteq W$ and $W'\cap W\subseteq V$ is not a special subspace, we know that $\ZW\cap\ZWp=Z(W'\cap W)$ is empty, and thus  $$[\ZW]\cdot[\ZWp]=0.$$
  \end{altenumerate}

Now we count the number of terms on the right hand side of \eqref{eq:complement} in first two cases.
\begin{altitemize}
\item The association $W\mapsto W^\perp$ gives a bijection between the set of codimension 1 special subspaces $W\subseteq V$
in Case (\ref{item:a})  and the set of isotropic $k$-lines in $W'^\perp$.  Hence the number of terms in Case (\ref{item:a}) is equal to $\frac{1-q^{2(d-1)}}{1-q^2}$, as $W'^\perp$ is a totally isotropic $k/\k$-hermitian space of dimension $d-1$.
\item The association $W\mapsto W^\perp$ gives a bijection between the set of codimension 1 special subspaces $W\subseteq V$ in Case (\ref{item:b}) and the set of isotropic $k$-lines in $W'\setminus W'^\perp$. Hence the number of terms in Case (\ref{item:b}) is equal to the number of vectors in $W'^\perp$ times the number of isotropic lines in the $k/\k$-hermitian space $W'/W'^\perp$. This is $q^{2(d-1)}(1+q^3)$, as $W'^\perp$ is of dimension $d-1$ and $W'/W'^\perp$ is of dimension 3.
\end{altitemize}
Thus the right hand side of \eqref{eq:complement} evaluates to $$\frac{1-q^2}{1+q^{2d+1}}\cdot\left( (1-q^2)\cdot \frac{1-q^{2(d-1)}}{1-q^2}+1\cdot q^{2(d-1)}(1+q^3)\right)=1-q^2,$$ which is equal to the left hand side by Lemma \ref{lem:c1ZW} applied to the 1-dimensional special cycle $\ZWp\subseteq\XV$.
\end{proof}

\begin{proposition}\label{prop:c1d}
  The following identity holds in $\Ch^d(\XVb)_\mathbb{Q}\simeq \mathbb{Q}$:
  \begin{equation*}\label{eq:c1d}
    c_1(\LX^*)^d=\prod_{i=1}^d(1-q^{2i}).
  \end{equation*}
\end{proposition}

\begin{proof}
  When $d=1$, this is \eqref{eq:fermat}. In general, we induct on $d$. By Proposition \ref{prop:c1}, we obtain
  \begin{equation}
    \label{eq:c1d}
    c_1(\LX^*)^d= \frac{1-q^2}{1+q^{2d+1}} \cdot c_1(\LX^*)^{d-1}\cdot\sum_{\codim W=1}[\ZW].
  \end{equation}
  By the projection formula and \eqref{eq:LXLWX}, we have (under the isomorphism \eqref{eq:ZWXW}), $$c_1(\LX^*)^{d-1}\cdot [\ZW]=j_{W,*}(c_1(\LXW^*)^{d-1}),$$  which evaluates to $\prod_{i=1}^{d-1}(1-q^{2i})$ by the induction hypothesis.

  The association $W\mapsto W^\perp$ gives a bijection the set of codimension 1 special subspaces $W\subseteq V$ and the set of isotropic $k$-lines in $V$. Thus the total number of terms in \eqref{eq:c1d} is the number of isotropic $k$-lines in the $(2d+1)$-dimensional $k/\k$-hermitian space $V$, which is $\frac{(1+q^{2d+1})(1-q^{2d})}{1-q^2}$ (e.g., by \cite[Lemma 1.9.1]{LZ}).  Thus \eqref{eq:c1d} evaluates to $$\frac{1-q^2}{1+q^{2d+1}}\cdot \prod_{i=1}^{d-1}(1-q^{2i})\cdot \frac{(1+q^{2d+1})(1-q^{2d})}{1-q^2}=\prod_{i=1}^{d}(1-q^{2i}).$$ This completes the proof.
\end{proof}

\subsection{Proof of Theorem \ref{thm:mainA}} \label{sec:proof-theor-refthm:m} Recall $\pl: \GrV\rightarrow \PV$ is the Pl\"ucker embedding.  By definition
\begin{equation}
  \label{eq:degdef}
  \deg\XV=[\XV]_{\PV}\cdot c_1(\mathcal{O}_{\PV}(1))^d
\end{equation}
is the intersection number of $\XV$ with $d$ general hyperplanes in $\PV$. By the projection formula, we obtain that $$\deg\XV=c_1((\pl\circ i)^*\mathcal{O}_{\PV}(1))^d$$ in $\Ch^d(\XV_\kb)_\mathbb{Q}\simeq \mathbb{Q}$. By the definition of the Pl\"ucker embedding, we have
\begin{equation*}
  \label{eq:pluckerbundle}
  \pl^*\mathcal{O}_{\PV}(1)\simeq\det \SGr^*.
\end{equation*}
 Thus $$c_1((\pl\circ i)^*\mathcal{O}_{\PV}(1))=i^*c_1(\det\SGr^*)=i^*c_1(\SGr^*).$$ By Lemma \ref{lem:bunGr} (\ref{item:b4}), we have $$i^*c_1(\SGr^*)=\frac{c_1(\LX^*)}{1-q},$$ and hence $$\deg\XV=\frac{c_1(\LX^*)^d}{(1-q)^d}.$$ The result then follows from Proposition \ref{prop:c1d}.

\section{Schubert calculus}

\subsection{Reminder on Schubert calculus (cf. \cite[Chapter 4]{EH16},\cite[\S14.7]{Ful98})} \label{sec:remind-schub-calc}
Let $0\le m\le n$. The Schubert classes of $\Gr_m(V)$ are indexed by $m$-tuples $a=(a_1,\ldots,a_m)$ of integers satisfying $$n-m\ge a_1\ge a_2\ge \cdots \ge a_m\ge0,$$ in other words, indexed by Young diagrams inside the $m\times (n-m)$ rectangles. For such an $m$-tuple $a$, define a \emph{Schubert cycle} $$\Sigma_a(V_\bullet):=\{U\in \Gr_m(V): \dim U\cap V_{n-m+i-a_i}\ge i,\ i=1,\ldots, m\}\subseteq \Gr_m(V),$$ where $$V_\bullet: 0\subset V_1\subset V_2\subset\cdots\subset V_{n-1}\subset V_{n}:=V$$ is a complete flag in $V$. The Schubert cycle $\Sigma_a(V_\bullet)\subseteq\Gr_m(V)$ is a closed subvariety of codimension $|a|:=\sum_{i=1}^m a_i$. Define the \emph{Schubert class} $$\sigma_a:=[\Sigma_a(V_\bullet)]\in \Ch^{|a|}(\Gr_m(V)_\kb),$$ which is independent of the choice of the complete flag $V_\bullet$. We use the standard notation suppressing trailing zeros in the indices. In particular, by definition $\sigma_{1}=\sigma_{1,0,\ldots,0}\in\Ch^1(\Gr_m(V)_\kb)$ is the hyperplane class under the Pl\"ucker embedding.

Define the \emph{dual} $a^*:=(n-m-a_m,\cdots, n-m-a_1)$. The Schubert classes form a $\mathbb{Q}$-basis of $\Ch^*(\Gr_m(V)_{\kb})_{\mathbb{Q}}$ and the intersection pairing $$\Ch^r(\Gr_m(V)_{\kb})_\mathbb{Q}\times \Ch^{m(n-m)-r}(\Gr_m(V)_{\kb})_\mathbb{Q}\rightarrow \mathbb{Q}$$ is perfect and has Schubert classes as dual basis, with $\sigma_a$ and $\sigma_b$ dual to each other if and only if $b=a^*$.

Define the \emph{conjugate} $a':=(a'_1,\ldots,a'_{n-m})$ such that $a'_j$ is the number of $i$'s such that $a_i\ge j$
. The Young diagrams of $a$ and $a'$ are mutual reflections along the main diagonal. The canonical isomorphism
\begin{equation}
  \label{eq:dualGr}
  \Gr_m(V)\simeq \Gr_{n-m}(V^*),\quad U\mapsto (V/U)^*
\end{equation}
 maps $\sigma_a\in \Ch^{|a|}(\Gr_m(V))$ to $\sigma_{a'}\in \Ch^{|a|}(\Gr_{n-m}(V^*))$.

\subsection{The class of $\XV$}

Let $\Grdd$ be the variety parametrizing partial flags $$U_\bullet: 0\subseteq U_d\subseteq U_{d+1}\subseteq V,$$ where $\dim U_d=d$ and $\dim U_{d+1}=d+1$. Define a closed embedding $$\psi: \Grdd\rightarrow \Gr_{d}(V)\times \Gr_{d+1}(V), \quad U_\bullet\mapsto (U_d, U_{d+1}).$$ Also consider the closed embedding $$(\phi,\id):\GrV\rightarrow \Grd\times \GrV,\quad U\mapsto (U^\perp, U).$$
Then by definition $(\phi,\id)$ induces an isomorphism between $\XV$ and $\im(\phi,\id)\cap \im(\psi)$.

\begin{proposition}\label{prop:imphi}
    The following identity holds in $\Ch^*(\Grd_\kb\times\GrV_\kb)_\mathbb{Q}$: $$[\im\psi]=\sum_{a,b} \sigma_{a^*}\times \sigma_{b^*},$$ where the sum runs over
  \begin{altitemize}
  \item $a=(a_1,\ldots,a_d)$ satisfying $d+1\ge a_1\ge\cdots \ge a_d\ge0$,
  \item $b=(b_1,\ldots, b_{d+1})$ satisfying $b_1=d$, $b_{i}=d+1-a_{d+2-i}\ge0$ for $i=2,\ldots,d+1$.
  \end{altitemize}

\end{proposition}

\begin{proof}
    Let $V_\bullet$ and $W_\bullet$ be two transverse complete flags in $V$ (\cite[Definition 4.4]{EH16}). Recall that $V_\bullet$ and $W_\bullet$ are transverse means that $V_i\cap W_{n-i}=0$ for any $i$. Let $$ V^{(i)}:=V_{n-d+i-a_i}, \quad W^{(i)}:=W_{n-(d+1)+i-b_i}.$$ Then by definition $$\Sigma_a(V_\bullet)=\{U_d\in\Grd: \dim U_d\cap V^{(i)}\ge i, \quad i=1,\ldots,d\}$$ and $$\Sigma_b(W_\bullet)=\{U_{d+1}\in\GrV: \dim U_{d+1}\cap W^{(i)}\ge i, \quad i=1,\ldots,d+1\}.$$ By the transversality it is easy to see that the intersection of $\im(\psi)$ with $\Sigma_a(V_\bullet)\times \Sigma_b(W_\bullet)\subseteq \Grd\times\GrV$ is nonempty if and only if $$\dim W^{(1)}=1, \quad \dim W^{(i)} \cap V^{(d+2-i)} =1, \quad i=2,\ldots d+1,$$ in which case the intersection is transverse at the unique point given by $$U_d=\bigoplus_{i=2}^{d+1}W^{(i)} \cap V^{(d+2-i)},\quad U_{d+1}= U_d \oplus W^{(1)}.$$ Therefore $$[\im \psi]\cdot (\sigma_a\times \sigma_b)=
  \begin{cases}
    1, & b_1=d,\quad b_i=d+1-a_{d+2-i}, i=2,\ldots, d+1,\\
    0, & \text{otherwise}.
  \end{cases}
  $$ The desired result then follows from the fact that Schubert classes form dual basis under the intersection pairing.
\end{proof}

\begin{corollary}\label{cor:fundclass}
      The following identity holds in $\Ch^*(\GrV_\kb)_\mathbb{Q}$: $$[\XV]=\sum_c \sigma_c\sigma_{\widehat c'} q^{|c|},$$ where the sum runs over $c=(c_1,\ldots,c_d)$ such that $d\ge c_1\ge \cdots \ge c_d\ge0$, and $\widehat c$ is the \emph{complement} of $c$ defined by $\hat c:=(d-c_d,\ldots, d-c_1)$.
\end{corollary}

\begin{proof}
  By definition we have $[\XV]=(\phi,\id)^*[\im(\psi)]$. Since $U^\perp\simeq \sigma^*(V/U)^*$, by (\ref{eq:dualGr}) it is easy to see that $$(\phi,\id)^*(\sigma_{a^*}\times\sigma_{b^*})= q^{|a^*|} \sigma_{(a^*)'}\sigma_{b^*} \in \Ch^*(\GrV_\kb)_\mathbb{Q}.$$
  Since $b=(d, d+1-a_d, d+1-a_2,\ldots, d+1-a_1)$, we know that $b^*=(a_1-1,\ldots,a_d-1,0)$. Let $c=(a_1-1,\ldots,a_d-1)$. Then $\widehat c=(d+1-a_d,\ldots, d+1-a_1)=a^*$.   It follows from Proposition \ref{prop:imphi} that $$[\XV]=\sum_{c} \sigma_{\widehat c'}\sigma_{c}q^{|\widehat c|}=\sum_{c} \sigma_{c}\sigma_{\widehat c'}q^{|c|}.$$ This completes the proof.\qedhere

\end{proof}

\begin{example}
  \begin{altitemize}
  \item    When $d=1$, by Corollary \ref{cor:fundclass} we obtain $$[\XV]=\sigma_1+\sigma_1q=(1+q)\sigma_1.$$ This agrees with the fact that $\XV\subseteq \GrV\simeq \mathbb{P}^2$ is the Fermat curve of degree $1+q$ (Example~\ref{exa:fermat}).
  \item When $d=2$, by Corollary \ref{cor:fundclass} we obtain $$[\XV]=\sigma_{2,2}+\sigma_1\sigma_{2,1}q+\sigma_2\sigma_{1,1}q^2+\sigma_{1,1}\sigma_{2}q^2+\sigma_{2,1}\sigma_1q^3+\sigma_{2,2}q^4.$$ 
  \end{altitemize}
\end{example}

\subsection{Proof of Theorem \ref{thm:mainB}}\label{sec:proof-theor-refthm:m-1} By (\ref{eq:degdef}) and the projection formula, we have $$\deg\XV=[\XV]\cdot c_1(\pl^*\mathcal{O}_{\PV}(1))^d.$$ The result then follows from Corollary \ref{cor:fundclass} and $\sigma_1=c_1(\pl^*\mathcal{O}_{\PV}(1))$ is the hyperplane place class under the Pl\"ucker embedding.

\subsection{Proof of Corollary \ref{cor:main}}\label{sec:proof-coroll-refc}
\begin{altenumerate}
\item This follows immediately from Theorems \ref{thm:mainA} and \ref{thm:mainB}.
\item Let $$S_c(x_1,\ldots,x_d)=\det(x_i^{c_j+d-j})/\det(x_i^{d-j})$$ ($1\le i,j\le d$) be the \emph{Schur polynomial} associated to $c=(c_1,\ldots,c_d)$ (\cite{Macdonald1992}). It is a symmetric polynomial of degree $|c|$. Write $$S:=\sum_c S_c(qx_1,\ldots, qx_d) S_{\widehat c'}(x_1,\ldots, x_d) S_1^d(x_1,\ldots, x_d)=:\sum_\lambda \kappa_{\lambda}S_\lambda,$$ where $\lambda$ runs over $\lambda=(\lambda_1,\ldots,\lambda_d)$ with $|\lambda|=d(d+1)$. By definition, the coefficient $\kappa_\lambda$ is given by the coefficient of $x_1^{\lambda_1+d-1}x_2^{\lambda_2+d-2}\cdots x_d^{\lambda_d}$ in
  \begin{equation}
    \label{eq:Schur}
    S\cdot \det(x_i^{d-j})=S\cdot \prod_{i<j}(x_i-x_j).
  \end{equation}
 Since the class of a point is $\sigma_{(d+1,\ldots,d+1)}$, we know that $$\sum_{c}\sigma_{c}\sigma_{\widehat c'}\sigma_1^dq^{|c|}=\kappa_{(d+1,d+1,\ldots, d+1)},$$ i.e., the coefficient of $x_1^{2d}x_2^{2d-1}\cdots x_d^{d+1}$ in (\ref{eq:Schur}). It remains to compute (\ref{eq:Schur}). By the dual Cauchy identity for Schur polynomials (\cite[0.11']{Macdonald1992}), we have $$\sum_cS_c(x_1,\ldots, x_d) S_{\widehat c'}(y_1,\ldots, y_d )=\prod_{i,j=1}^d(x_i+y_j).$$ By definition, we have $$S_1(x_1,\ldots, x_d)=x_1+\cdots+x_d.$$ Therefore $$S=\sum_c S_c(qx_1,\ldots, qx_d) S_{\widehat c'}(x_1,\ldots, x_d) S_1^d(x_1,\ldots, x_d)=\left(\prod_{i,j=1}^d(qx_i+x_j)\right)(x_1+\cdots+x_d)^d,$$ and thus
\begin{equation}
  \label{eq:coeffpoly}
  S\cdot \prod_{i<j}(x_i-x_j)=\left(\prod_{i,j=1}^d(qx_i+x_j)\right)(x_1+\cdots+x_d)^d\left(\prod_{i<j}(x_i-x_j)\right). 
\end{equation} The result then follows from Item (\ref{item:m0}).
\item    By applying Pieri's formula (\cite[Proposition 4.9]{EH16}) $d$ times, we know that the term $\sigma_c\sigma_{\widehat c'}\sigma_1^d$ is equal to the number of sequences of Young diagrams $c^{(0)},c^{(1)}\ldots, c^{(d)}$ starting with $c^{(0)}=c$ and ending with $c^{(d)}=(\widehat c')^*$
such that each $c^{(i+1)}$ has exactly one more box than $c^{(i)}$. Equivalently, it is the number of standard Young tableaux of skew shape $(\widehat c')^*/c$. Now Item (\ref{item:m0}) shows that the number of such standard Young tableaux with $|c|=\Ell$ is equal to the coefficient of $q^\Ell$ in $$\prod_{i=1}^d{\frac{1-q^{2i}}{1-q}}=\prod_{i=1}^d\sum_{j=0}^{2i-1}q^j,$$ which is equal to the number of ordered partitions $(\Ell_1,\ldots, \Ell_d)$ of $\Ell$ satisfying the extra conditions $$0\le \Ell_i\le 2i-1, \quad i=1,\ldots,d.$$
\end{altenumerate}

\begin{example}\label{exa:smalld} We end with an example illustrating Corollary \ref{cor:main} (\ref{item:m2}) (\ref{item:m1}).
  \begin{altitemize}
  \item  When $d=1$, the polynomial (\ref{eq:coeffpoly}) is equal to $(1+q) x_1^2$. The coefficient of $x_1^2$  is given by $1+q=\frac{1-q^2}{1-q}$ as in Corollary \ref{cor:main} (\ref{item:m2}).
  \item  When $d=2$, the polynomial (\ref{eq:coeffpoly}) is equal to
   \begin{align*}
     \left(q+2 q^2+q^3\right) x_1^6 x_2+\left(1+3 q+4 q^2+3 q^3+q^4\right) x_1^5 x_2^2+\left(1+2 q+2 q^2+2 q^3+q^4\right) x_1^4 x_2^3\\+\left(-1-2 q-2 q^2-2 q^3-q^4\right) x_1^3 x_2^4+\left(-1-3 q-4 q^2-3 q^3-q^4\right) x_1^2 x_2^5+\left(-q-2 q^2-q^3\right) x_1 x_2^6
   \end{align*} The coefficient of $x_1^4x_2^3$ is given by $$1+2 q+2 q^2+2 q^3+q^4=(1+q)(1+q+q^2+q^3),$$ which equals $\frac{1-q^2}{1-q}\cdot\frac{1-q^4}{1-q}$ as in Corollary \ref{cor:main} (\ref{item:m2}).
 \item When $d=3$, the coefficient of $x_1^6x_2^5x_3^4$ is equal to $$1 + 3 q + 5 q^2 + 7 q^3 + 8 q^4 + 8 q^5 + 7 q^6 + 5 q^7 + 3 q^8 + q^9= (1+q)(1+q^2+q^3)(1+q+q^2+q^3+q^4+q^5),$$ as in Corollary \ref{cor:main} (\ref{item:m2}). In Table \ref{tab:skew} we list all standard Young tableaux of skew shape $(\widehat c')^*/c$ with $|c|=4$ as in Corollary \ref{cor:main} (\ref{item:m1}). Notice the total number of such Young tableaux  is 8, which indeed agrees with the coefficient of $q^4$.

   \begin{table*}[h]
     \centering
     \begin{tabular}[h]{|c|c|c|}
       $c$ & $(\widehat c')^*$ & $(\widehat c')^*/c$\\[1em]
       $\yng(3,1)$ &  $\yng(3,2,1,1)$ & $\young(:1,2,3)$\qquad $\young(:2,1,3)$\qquad $\young(:3,1,2)$\\[3em]
       $\yng(2,2)$ & $\yng(3,2,2)$ & \begin{ytableau}
         \none & \none  & 1\\
         \none & \none  & \none\\
         2 & 3
\end{ytableau}\qquad \begin{ytableau}
         \none & \none  & 2\\
         \none & \none  & \none\\
         1 & 3
\end{ytableau}\qquad \begin{ytableau}
         \none & \none  & 3\\
         \none & \none  & \none\\
         1 & 2
\end{ytableau} 

\\[3em]
       $\yng(2,1,1)$ & $\yng(3,3,1)$ & $\young(:1,23)$\qquad $\young(:2,13)$\\[2em]
     \end{tabular}
\caption{standard Young tableaux of skew shape $(\widehat c')^*/c$}\label{tab:skew}
   \end{table*}
   \end{altitemize}
\end{example}

\bibliographystyle{alpha}
\bibliography{KR}

\end{document}